\documentclass[12pt]{article}
\usepackage[reqno]{amsmath}
\usepackage{amssymb, amsthm, enumerate}
\usepackage{graphicx}
\usepackage{fullpage}
\usepackage{tikz}
\usetikzlibrary{decorations.markings}
\usepackage{array,booktabs,hhline}
\usepackage{float}
\usepackage{wrapfig}
\usepackage{url,hyperref}
\usepackage{longtable}

\usepackage[backend=biber,style=numeric,doi=false,isbn=false,url=false,hyperref=true,backref=true,maxbibnames=10]{biblatex}
\bibliography{qwalk.bib}

\expandafter\ifx\csname pdfpagebox\endcsname\relax\else
\fi

\makeatletter
\newtheorem*{rep@theorem}{\rep@title}
\newcommand{\newreptheorem}[2]{%
\newenvironment{rep#1}[1]{%
 \def\rep@title{#2 \ref{##1}}%
 \begin{rep@theorem}}%
 {\end{rep@theorem}}}
\makeatother

\newtheoremstyle{plainsl}%
	{\topsep}
	{\topsep}
	{\slshape} 
	{}
	{\normalfont\bfseries}
	{.}
	{ }
	{}

\swapnumbers

{\theoremstyle{plainsl}
\newtheorem{theorem}{Theorem}[section]
\newtheorem{lemma}[theorem]{Lemma}
\newtheorem{openprob}{Open Problem}[section]
\newtheorem{corollary}[theorem]{Corollary}}
{\theoremstyle{remark}
}

\newreptheorem{theorem}{Theorem}

\newcommand\cref[1]{Corollary~\ref{cor:#1}}

\def\proof{\noindent{{\sl Proof. }}}
\def\sqr#1#2{{\vbox{\hrule height.#2pt
    \hbox{\vrule width.#2pt height#1pt \kern#1pt
        \vrule width.#2pt}\hrule height.#2pt}}}
\def\eqed{\sqr53}
\def\qed{%
    \ifmmode\eqno\eqed
    \else\nobreak\ \hfill\eqed\medbreak\fi}

\newcommand\cM{{\mathcal M}}

\newcommand\Zv{{\mathbf v}}
\newcommand\Ze{{\mathbf e}}
\newcommand\ee{{\mathrm e}}
\newcommand\ii{{\mathrm i}}

\newcommand\cx{{\mathbb C}}

\newcommand\re{{\mathbb R}}

\newcommand\comp[1]{{\mkern2mu\overline{\mkern-2mu#1}}}

\newcommand\pmat[1]{\begin{pmatrix} #1 \end{pmatrix}}

\newcommand\mxm{\widehat{M}}

\date{April 2, 2024}
\title{Selected Open Problems in \\Continuous-Time Quantum Walks}

\author{Gabriel Coutinho\thanks{Universidade Federal de Minas Gerais, Belo Horizonte, Brazil. \protect\url{gabriel@dcc.ufmg.br}} \and Krystal Guo\thanks{Korteweg-de Vries Institute for Mathematics, University of Amsterdam, Amsterdam, The Netherlands. QuSoft (Research center for Quantum software \& technology), Amsterdam, The Netherlands. \protect\url{k.guo@uva.nl}}}

\begin{document}
\maketitle

\begin{abstract}
Quantum walks on graphs are fundamental to quantum computing and have led to many interesting open problems in algebraic graph theory. This review article  highlights three key classes of open problems in this domain; perfect state transfer, instantaneous uniform mixing, and average mixing matrices. In highlighting these open problems, our aim is to stimulate further research and exploration in this rapidly evolving field.
\end{abstract}

\section{Introduction}\label{sec:intro}


Quantum walks on graphs are essential for quantum computing. Amongst other important advances, quantum walks were shown to be a  universal computational primitive in a landmark result of Childs in \cite{ChildsUniversalQComputation}. They have also gained prominence in algebraic graph theory; many problems arising in the study of quantum walks have given rise to interesting combinatorial problems. There are many open problems in this rapidly evolving field, and in this review article, we highlight three classes of problems, place them in context with relevant results and propose several unsolved problems. 

Most of the results stated in this text are either not new or follow easily from known facts, but there are some new theorems. These are not the  focus of this paper; our main goal is to survey the basics and propose problems we believe are central to the field. 
We refer the reader to \cite{coutinho2018continuous} for a full survey of the basic result. For the purposes of this paper, we give a very brief introduction, as follows. 

The \textit{transition matrix} of a \textsl{continuous-time quantum walk} on a graph $G$ with adjacency matrix $A$ is a matrix-valued function in time, denoted $U(t)$, and is given as follows:
\[ U(t) = \exp(\ii tA).
\]
The quantum state of the quantum system defined on the graph at time $t$ is given by $U(t)$ applied to the vector describing the initial state. Typically we will require the initial state or the target state to be either $\Ze_u$, the characteristic vector of a vertex $u \in V(G)$, or for it to have entries of constant absolute value.

The three classes of problems are discussed are summarized here. 

\begin{enumerate}[(1)]
    \item State transfer occurs when the state at one vertex is transferred to another vertex, either perfectly or approximately. Its study is partially motivated by the No-Cloning Theorem, which is a seminal result in quantum mechanics which forbids the duplication of an arbitrary unknown quantum state. Many aspects of state transfer have been studied, see for instance \cite{ChrDatEke2004, KayReviewPST,God2010,CoutinhoGodsilGuoVanhove2,CouGuovBo2017,KemLipTau2017, GodGuoKem2020} for a very limited sample of papers on the topic.
    \item Instantaneous uniform mixing occurs when all entries of $U(t)$ have the same absolute value. In layman's terms, no matter where you start your quantum walk, after this special time $t$, the probability of being anywhere is constant. This has applications in entanglement generation and state preparation, and has been studied in \cite{adamczak2007non,GodsilMullinRoy,AdaChanComplexHadamardIUMPST}. We should warn that uniform mixing is a considerably less understood concept than state transfer.
    \item We also consider the average of $M(t)$ on the real line (yes, this is ill-defined, but we consider the limit of the uniform averages of $M(t)$ on arbitrarily long intervals, and this is allowed). This matrix satisfies several pleasant linear algebraic properties, and some of these connect very naturally to the combinatorics of the graph. It has been studied in notably in \cite{adamczak2003note,GodsilAverageMixing, CouGodGuo2018,GodGuoSin2018}.
\end{enumerate}


\section{Economic perfect state transfer}


If we think of edges in a graph as corresponding to interactions between qubits, then, for the purposes of implementation, it would be more economical to minimize the number of edges, while preserving the behaviour. In this vein, one might ask for the graph on $n$ vertices with the least number of edges such that there is perfect state transfer between some two vertices. This question was first asked by Godsil in \cite{GodsilPeriodicGraphs11} but remains open and has motivated other works with partial results \cite{coutinho2019quantum,kay2018perfect}. 

Henceforth in this paper, we will denote the spectral decomposition of the adjacency matrix $A(G)$ by
\[
A = \sum_{i = 0}^{d} \theta_r E_r.
\]

\subsection{Lower bound}

Formally, perfect state transfer occurs at a $G$ between vertices $u$ and $v$ at time $t$ if there is a complex number $\gamma$ of absolute value $1$ so that
\[
U(t) \Ze_u = \gamma \Ze_v,
\]
where $\Ze_u$ and $\Ze_v$ are the characteristic vectors of these vertices. It is easy to verify that if perfect state transfer occurs, then for all $r$, $E_r \Ze_u = \pm E_r \Ze_v$. Whenever these vectors are nonzero, and depending on these signs, $\ee^{\ii \theta_r t} = \pm \gamma$, which leads to a strong restriction on the $\theta_r$. In particular,
\[
t (\theta_r - \theta_s) = k_{rs} \pi,
\]
where $k_{rs}$ is an integer of suitable parity. We refer the reader to \cite{coutinho2018continuous} for a more developed introduction to the topic. 

In fact, \cite[Theorem 14]{coutinho2023spectrum} establishes that the minimum time perfect state transfer occurs is at most $\pi / \sqrt{2}$, therefore for those eigenvalues $\theta_r$ for which $E_r \Ze_u \neq 0$, the difference between any two of them is at least $\sqrt{2}$. 

Following the steps in \cite{coutinho2019quantum}, the number of distinct eigenvalues for which $E_r \Ze_u \neq 0$ upper bounds $(D/2)+1$, where $D$ is the diameter of the graph, and they are all separated by at least $\sqrt{2}$. The sum of the squares of all eigenvalues of the graph is $2m$, where $m$ is number of edges, and therefore an upper bound to the diameter can be obtained in terms of the number of edges. The result in \cite{coutinho2019quantum} is stated for periodic vertices and only assumes the eigenvalues are separated by $1$. Recalling that if perfect state transfer occurs, then at double the time there is periodicity, and using the recent result that improves separation to $\sqrt{2}$, the same steps of \cite{coutinho2019quantum} result in the following bound for the diameter $D$ of graphs on $m$ edges that admit perfect state transfer:
\[
D^3 \leq 80 m.
\]
The obvious immediate problem is to improve this.
\begin{openprob} \label{lowerbound}
   If perfect state transfer occurs in a graph with $m$ edges of diameter $D$, show that
   \[
    m \in \Omega(D^k),
   \]
   for $k > 3$.
\end{openprob}
The result in \cite{coutinho2019quantum} mentioned above does not exploit the fact that vertices involved in perfect state transfer are strongly cospectral. The theory developed in \cite{coutinho2023spectrum} might provide some quantitative ways of exploiting this fact.

Standard eigenvalue bounds for the diameter, for instance the well-known results in \cite{chung1989diameters}, do not seem to be particularly useful in our context, however we have no strong reason to believe they are not.

The strongest reason we have to believe it is possible to provide some improvement in Problem~\ref{lowerbound} is the fact that our best examples are exponentially large in terms of $D$, as we expose in the next section. The strongest reason we have to believe it might not be possible to improve the cubic bound is the following family of graphs, found by Dragan Stevanović (private communication):
\begin{figure}[H] 
    \begin{center}
    \includegraphics[scale=1]{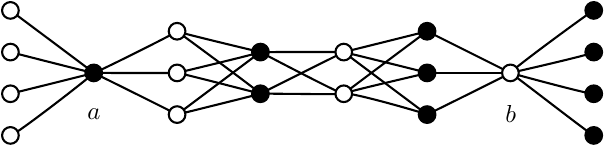}
    \caption{One specimen of a family found by Dragan Stevanović.}
    \end{center}
\end{figure}

Hopefully the white-black color coding indicates how to generate the rest of the family straightforwardly. Vertices $a$ and $b$ are strongly cospectral, and the eigenvalues for which $E_r \Ze_a \neq 0$ are the integers in $[-n,n]$, except for $0$, where $n$ is the number of vertices in the first layer. In a sense, they almost admit perfect state transfer between them, but they do not, for any $n$. And the number of edges in this family grows cubic in the diameter, as can be easily verified.

\subsection{Upper bound}

The standard examples of graphs admitting perfect state transfer at arbitrarily large distances are the hypercubes $Q_n$ (see for instance \cite{ChristandlPSTQuantumSpinNet2}), which we like to view as the iterared Cartesian product of the graph $P_2$ with itself. Here when the diameter is $D$, the number of vertices is $2^D$ and the number of edges is $D \times 2^{D-1}$. Iterated products of $P_3$ with itself also provide an infinite family slightly sparser: $3^D$ vertices, $D \times 3^{D/2 -1}$ edges (the relevant result here is that when $G$ admits perfect state transfer at time $t$, then any Cartesian power of $G$ also does, and at the same time.)

An operation inspired by the works in \cite{fiol1996locally,bachman2011perfect} allows for some local improvement. Here we apply the well-known theory of equitable partitions (see \cite[Chapter 9]{GodsilRoyle} for instance) to quantum walks, with a twist.

If $P$ the characteristic matrix of a partition $\pi$ of the vertex set of $G$, then $\pi$ is equitable when the column space of $P$ is $A$-invariant, where $A = A(G)$. In this case, there is a square matrix $B$ so that
begin
\[
    A P = PB.
\]
whence $\exp(\ii t A)P = P \exp(\ii t B)$. If the columns of $P$ are normalized, then $P^T A P = B$, so $B$ is symmetric, and $\exp(\ii t B)$ is the transition matrix of the quantum walk defined on the weighted graph whose vertex set corresponds to the rows and columns of $B$. If columns $i$ and $j$ of $P$ are equal to the characteristic vectors of two vertices from $G$, say $\Ze_u$ and $\Ze_v$, then it follows that there is perfect state transfer between $u$ and $v$ in $G$ if and only if there is perfect state transfer between $i$ and $j$ in the graph defined by $B$. This is essentially a result from \cite{bachman2011perfect}. The twist now comes from the fact that $P$ needs not be the normalized characteristic matrix of a partition. Any matrix with orthonormal columns will do, so long as its column space is $A$-invariant, and two of its columns are equal to $\Ze_u$ and $\Ze_v$. The concept of pseudo-equitable partition \cite{fiol1996locally} comes in handy, as the columns of $P$ still carry some combinatorial meaning.

\begin{theorem} \label{thm:quotient}
    Let $B$ be a symmetric matrix so that for some time $t$ and $\lambda\in \cx$, 
    \[
        \exp(\ii t B) \Ze_i = \lambda \Ze_j.
    \]
    Let $P$ be the characteristic matrix of a partition of $V(G)$, with $u,v \in V(G)$, and $D$ a diagonal matrix, so that $P\Ze_i = \Ze_u$, $P\Ze_j = \Ze_v$, and $D_{uu} = D_{vv} = 1$. Assume the column space of $DP$ is $A$-invariant, let $S$ be obtained from $DP$ upon normalizing its columns, and finally assume that $AS = SB$. Then $G$ admits perfect state transfer between $u$ and $v$.
\end{theorem}
\proof
    From $AS = SB$ we have
    \[
        \exp(\ii t A) S = S \exp(\ii t B).
    \]
    Multiplying both sides by $\Ze_i$, we have
    \[
        \exp(\ii t A) S \Ze_i = \lambda S \Ze_j \implies \exp(\ii t A) \Ze_u = \lambda \Ze_v.
    \]
\qed

This simple proof hides a procedure to construct new graphs admitting perfect state transfer, based on hypercubes, but with fewer vertices and edges, which we exemplify below for $Q_4$:

\begin{figure}[H] 
    \begin{center}
    \includegraphics[scale=1]{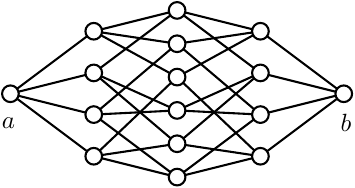}
    \caption{The graph $Q_4$.}
    \end{center}
\end{figure}

As we commented, there is perfect state transfer between $a$ and $b$. The distance partition with respect to $a$ is equitable. Assume its characteristic matrix is $P_1$. The symmetric matrix $B$ which represents the action of $A(Q_4)$ onto the columns of its normalized partition matrix $S_1$ is the adjacency matrix of the following weighted path:
\begin{figure}[H] 
    \begin{center}
    \includegraphics[scale=1]{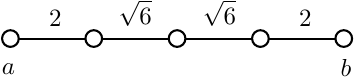}
    \caption{The symmetric quotient of the distance partition in $Q_4$.}
    \end{center}
\end{figure}
Again, there is perfect state transfer between $a$ and $b$ (simply apply Theorem~\ref{thm:quotient} with $P_1 = P$, $S_1 = S$ and $D = I$). Consider now the graph $G$ below:
\begin{figure}[H] 
    \begin{center}
    \includegraphics[scale=1]{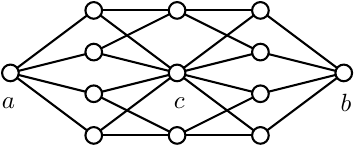}
    \caption{The graph $G$ which is a sort of modified $Q_4$, obtained upon ``compressing" four vertices.}
    \end{center}
\end{figure}
If $D$ is the diagonal matrix that assigns $1$ to all vertices and $2$ to vertex $c$, and if $P_2$ is the characteristic matrix of the distance partition of $a$, then the column space of $DP_2$ is $A$-invariant. If $S_2$ is obtained upon normalizing the columns of $DP_2$, then it is a straightforward calculation to verify that the same matrix $B$ as above represents the action of $A(G)$ onto the columns of $S_2$. Thus Theorem~\ref{thm:quotient} applies, and $G$ admits perfect state transfer between $a$ and $b$.

There is a combinatorial interpretation of a partition being equitable. It consists of verifying that the out degree of every vertex from class $i$ to class $j$ is constant, depending only on $i$ and $j$. It is straightforward to verify that the distance partition of $a$ in $G$ above also satisfies this regularity, except that edges arriving at $c$ should be counted with weight $2$, and edges departing $c$ should be counted with weight $1/2$. This combinatorial interpretation was heavily exploited in \cite{kay2018perfect} to generate modified versions of several Cartesian powers of $P_2$ and $P_3$, each with significantly fewer vertices and edges than the original. The procedure however did not lead to an infinite family whose size of each graph was bounded by a polynomial in the diameter. 

It is plausible that there is a subfamily of Cartesian powers of $P_2$ for which this compressing procedure can be described in such a way that the resulting graphs are bounded 
by a polynomial in the diameter. 

\begin{openprob}
    Find an infinite family of graphs for which there is perfect state transfer, and the size of the graph is bounded by a polynomial in the diameter.
\end{openprob}

\subsection{Bonus problems related to state transfer}

Godsil and Smith \cite{godsil2024strongly} provide an extensive theory of strongly cospectral vertices. The standard characterization says that two vertices are strongly cospectral if and only if they are cospectral and parallel. Vertices $u$ and $v$ are cospectral if, for any fixed $k$, the number of closed walks of length $k$ starting and ending in $u$ is the same as the number for $v$. It is fair to say this is a combinatorial characterization of cospectrality.
\begin{openprob}
    Find a combinatorial characterization of strong cospectrality.
\end{openprob}

It could be that there is no perfect state transfer between $u$ and $v$, which would correspond to $|\exp(\ii t A)_{uv}|$ being equal to $1$ for some $t$, but there could be distinct values of $t$ for which the sequence $(|\exp(\ii t_j A)_{uv}|)_{j \geq 0}$ converges to $1$. When this occurs, we say that pretty good state transfer happens in $G$ between $u$ and $v$. One of the results in \cite{coutinho2023irrational} says that it is possible to verify whether a given graph admits pretty good state transfer, but it does not lead to a procedure to compute the sequence of times $(t_j)$. It is also remarkable that perfect state transfer can be tested in polynomial time, whereas the algorithm in \cite{coutinho2023irrational} presents exponential complexity.

\begin{openprob}
    Find a polynomial time algorithm to test whether pretty good state transfer occurs, and find an algorithm to compute a subsequence of times which represent pretty good state transfer.
\end{openprob}

\section{Uniform mixing}

There is a surprising connection between uniform mixing and perfect state transfer. 
If the vertices of a graph $G$ can be partitioned into pairs such that there is perfect state transfer between each pair at some time $\tau$, then the transition matrix at that time, $U^{\tau}$, must be a permutation matrix and thus have eigenvalues which are roots of unity. In all known cases of uniform mixing, the eigenvalues of $U^{\tau}$, where $\tau$ is the time of instantaneous uniform mixing, are also roots of unities, but there is no simple combinatorial explanation. This is the topic of our next open problem.


Let $G$ be a graph with adjacency matrix $A$. Recall that the transition matrix of the continuous-time quantum walk on $G$ is given by 
\[
U(t) =  \exp{\ii tA}
\]
and the mixing matrix at time $t$ is given by 
\[
M(t) = U(t) \circ \comp{U(t)} = U(t) \circ U(-t).
\]
We say that $G$ admits \textsl{instantaneous uniform mixing} (or simply \textsl{uniform mixing}) at time $\tau$ if all entries of $M(\tau)$ are equal. Equivalently, $G$ admits  uniform mixing at time $\tau$ if all entries of $U(\tau)$ have the same absolute value.


\begin{openprob}
    If $G$ admits uniform mixing at time $\tau$, then the eigenvalues of $U(\tau)$ are roots of unity. 
\end{openprob}


Uniform mixing has been studied in special graph classes, including  in cycles
\cite{adamczak2007non},  in circulant graphs  \cite{TamonAhmadiUniformMixingCirculants, adamczak2003note} and in association schemes in \cite{AdaChanComplexHadamardIUMPST, GodsilMullinRoy} 

\subsection{Uniform mixing in association schemes}


We consider a continuous-time quantum walk on $G$, a distance regular graph of diameter $d$. Instead of looking at the eigenvalues of $U(t)$, we can look at the mixing matrix 
\[
M(t) = U(t) \circ \comp{U(t)}
\]
which is symmetric and doubly-stochastic with non-negative real entries. We can give an explicit formula for the eigenvalues of $M(t)$ in terms of the $P$ and $Q$ matrices of the distance regular graph. 


Recall that since $\{E_i\}_{i=0}^d$ and $\{A_i\}_{i=0}^d$ are both bases of $\cM$, there exists change of bases matrices between them. The \textsl{eigenmatrices} of the association scheme are  $(d+1) \times (d+1)$ matrices $P$ and $Q$ such that
\[
A_j = \sum_{i = 0}^d P_{ij}E_i \text{ and } E_j = \frac{1}{n} \sum_{i=0}^d Q_{ij} A_i,
\]
where $n$ is the number of vertices.
Note that this implies that $\{P_{ij}\}_{i=0}^d$ are the eigenvalues of $A_j$.

\begin{lemma}\label{lem:eigs} If $G$ is a distance-regular graph of diameter $d$ and $M(t)$ is the mixing matrix, then
\begin{equation}\label{eigM}
 M(t) =  \sum_{\ell = 0}^d \frac{1}{n^2} \left(\sum_{s,r,r'}  \ee^{\ii tP_{r1}-itP_{r'1}}  Q_{sr}  Q_{sr'}P_{\ell s}  \right)   E_{\ell} 
\end{equation}
for $\ell = 0, \ldots, d$. 
\end{lemma}

\proof Since $G$ is a distance-regular, $A(G) = A_1$ in an association. The eigenvalues of $A_1$ are $\{P_{r1}\}_{r=0}^d$ and we see that 
\[
U(t) =\sum_{r=0}^d \ee^{\ii tP_{r1}} E_r, \quad \comp{U(t)} =\sum_{r=0}^d \ee^{-\ii tP_{r1}} E_r.
\]
Since we want to take a Schur product, we will write $U(t), \comp{U(t)}$ in the Schur basis  as follows: 
\[
U(t) =\sum_{r=0}^d \ee^{\ii tP_{r1}} \left(  \frac{1}{n} \sum_{s=0}^d Q_{sr} A_s \right) =   \frac{1}{n} \sum_{s=0}^d\left(\sum_{r=0}^d  \ee^{\ii tP_{r1}}  Q_{sr} \right) A_s 
\]
whence
\[
\begin{split}
 M(t)=   U(t)\circ \comp{U(t)} &= \left( \frac{1}{n} \sum_{s=0}^d\left(\sum_{r=0}^d  \ee^{\ii tP_{r1}}  Q_{sr}\right)  A_s  \right) \circ \left( \frac{1}{n} \sum_{s'=0}^d\left(\sum_{r'=0}^d  \ee^{-\ii tP_{r'1}} Q_{s'r'} \right) A_{s'}  \right) \\
    &= \frac{1}{n^2} \sum_{s=0}^d\left(\sum_{r=0}^d  \ee^{\ii tP_{r1}}  Q_{sr}\right)\left(\sum_{r'=0}^d  \ee^{-\ii tP_{r'1}} Q_{sr'} \right)  A_s \\
    &= \frac{1}{n^2} \sum_{s=0}^d\left(\sum_{r,r'}  \ee^{\ii tP_{r1}-itP_{r'1}}  Q_{sr}  Q_{sr'} \right)  A_s .
\end{split}
\]
To find the eigenvalues, we will change into the other basis. 
\[
\begin{split}
 M(t) &= \frac{1}{n^2} \sum_{s=0}^d\left(\sum_{r,r'}  \ee^{\ii tP_{r1}-\ii tP_{r'1}}  Q_{sr}  Q_{sr'} \right)  \left(\sum_{\ell = 0}^d P_{\ell s}E_{\ell} \right) \\
 &= \sum_{\ell = 0}^d \frac{1}{n^2} \left(\sum_{s,r,r'}  \ee^{\ii tP_{r1}-itP_{r'1}}  Q_{sr}  Q_{sr'}P_{\ell s}  \right)   E_{\ell} ,
\end{split}
\]
which gives the eigenvalues of $M(t)$. \qed 

The following lemma will allow us to determine if the distance regular graph has perfect state transfer between pairs of antipodal vertices from the eigenvalues of $M(t)$. 

\begin{lemma}\label{lem:J} If $M$ is a symmetric $n\times n$  matrix such that the following hold:
\begin{enumerate}[(a)]
\item $M$ is doubly stochastic, and,
\item the characteristic polynomial of $M$ is $x^{n-1}(x-1)$,
\end{enumerate}
then $M = \frac{1}{n}J$. \end{lemma}
\proof 
Suppose $M$ is a $n\times n$ matrix satisfying the hypotheses in the theorem. Since $M$ has exactly one non-zero eigenvalue, then $M$ has exactly one linear independent row. We may write
\[
M= \pmat{\Zv \\ \alpha_1 \Zv \\ \vdots \\ \alpha_{n-1} \Zv }
\]
where the $\alpha_j$s are scalars in $\re$. Since row sums are constant, we have 
\[ \sum_{\ell = 1}^n \Zv_{\ell} = \alpha_j \sum_{\ell = 1}^n \Zv_{\ell}
\]
for $j = 1, \ldots, n-1$. Then $\alpha_j = 1$ for $j = 1, \ldots, n-1$.

Suppose there exists $j$ and $ \ell$ in $\{1, \ldots, n\}$ such that $\Zv_{j} \neq \Zv_{\ell}$. Since $M$ is symmetric, we have that 
\[
\Zv_j = M(1,j) = M(j,1) = \Zv_1
\]
and 
\[
\Zv_{\ell} = M(1,\ell) = M(\ell,1) = \Zv_1 = \Zv_j,
\]
a contradiction. Thus $\Zv$ is a constant vector. Since $\Zv$ has non-negative entries and sums to 1, we have that 
$M = \frac{1}{n} J$ as required.\qed 

If we combine this with the previous lemma, we obtain a necessary and sufficient condition for uniform mixing, using the parameters of the graph. 

\begin{corollary}\label{cor:um} If $A(G)$ is a relation of an association scheme, then $G$ admits uniform mixing at time $t$ if and only if 
\[ 
\frac{1}{n^2} \left(\sum_{s,r,r'}  \ee^{\ii tP_{r1}-\ii tP_{r'1}}  Q_{sr}  Q_{sr'}P_{\ell s}  \right) = \begin{cases}
     1, &\text{ if } \ell = 0; \\
     0, &\text{otherwise}. 
 \end{cases}
 \]
\end{corollary}
\begin{proof} Follow directly from Lemmas \ref{lem:eigs} and \ref{lem:J}.\qed
\end{proof}

Though this appears as if it would be difficult to use, however, it does give us a way to rule out uniform mixing for some classes of graphs. For a Hadamard graph:
\[P=Q=\pmat{
  1 & n & 2 n -2 & n & 1 \\
1 &\sqrt{n}& 0&-\sqrt{n}&-1 \\
1 & 0&-2&0&1\\
1& - \sqrt{n} & 0& \sqrt{n}&-1 \\
1  & -n  & 2 n -2& -n & 1}.\]

If $n$ is a square, the  Hadamard graph of order $n$ has perfect state transfer at time  $t = \frac{\pi}{\sqrt{n}}$. (If $n$ is not a square, then the graph has non-integral eigenvalues.) 

\begin{lemma} The Hadamard graph of order $n$ has uniform mixing at some time $t$  if and only if $n =4$. \end{lemma}

\proof We will look at the $\ell=2$ eigenvalue of $M(t)$ for any time $t$.  We get 
\[
\begin{split}
\lambda_2 &= \frac{1}{n^2} \left( 2  n^{2} \ee^{\left(2 \ii   \sqrt{n} t\right)} + 2  n^{2} \ee^{\left(-2 \ii   \sqrt{n} t\right)} + 12  n^{2} + 8  n \ee^{\left(\ii   n t\right)} + 8  n \ee^{\left(-\ii   n t\right)} - 16  n \right) \\
 &= 2  \ee^{\left(2 \ii   \sqrt{n} t\right)} + 2 \ee^{\left(-2 \ii   \sqrt{n} t\right)} + 12  + \frac{8}{n}   \ee^{\left(\ii   n t\right)} + \frac{8}{n} \ee^{\left(-\ii   n t\right)} - \frac{16}{n}  
\end{split}
\]
If $\lambda_2= 0$, then we have 
\[\begin{split}
2 \left( \ee^{\left(2 \ii   \sqrt{n} t\right)} + \ee^{\left(-2 \ii   \sqrt{n} t\right)} \right) + \frac{8}{n}  \left( \ee^{\left(\ii   n t\right)} +  \ee^{\left(-\ii   n t\right)}\right) &=  \frac{16}{n} -12 \\
4 \left( \cos (2  \sqrt{n}t )  \right) + \frac{16}{n}  \left( \cos \left(n t\right)\right) &=  \frac{16}{n} -12 \\
4 \cos (2  \sqrt{n}t )  + 12  &=  \frac{16}{n} -\frac{16}{n}  \cos \left(n t\right) \\
4 \left( \cos (2  \sqrt{n}t )  + 3\right)  &=  \frac{16}{n} \left( 1 -  \cos \left(n t\right) \right). \\
\end{split}
\]
We see that the left hand side is at least 8 for any $n,t$. Then 
\[
8\leq \frac{16}{n}(1-\cos(nt) ) \leq \frac{32}{n}
\]
from whence we obtain that $n \leq 4$. We can verify that the Hadamard graph with $n =2$ does not have uniform mixing but the Hadamard graph with $n=4$ does uniform mixing at time $t = \frac{\pi}{\sqrt{n}}$, since it is isomorphic to the hypercube on $16$ vertices. \qed

This leads us to our next open problem. Like the Hadamard graph, $C_9$ is also a formally self-dual distance-regular graph of diameter $4$. Unfortunately, it seems that the application of Corollary \ref{cor:um} is much more diffcult in this case. 

\begin{openprob}
   Determine if  $C_9$ admits uniform mixing. 
\end{openprob}

 We note that it is possible that  $C_9$ admits uniform mixing but the eigenvalues of $U^{\tau}$ are not roots of unity. The open problem first appeared in \cite{GodsilMullinRoy} and in the thesis \cite{NatalieThesis} of one of the authors, Natalie Mullin. Their work shows that uniform mixing does not occur on even cycle and nor on any cycle of prime length. 

A major drawback of Corollary \ref{cor:um} is that, unlike the characterisation of perfect state transfer, it is not independent of the time of uniform mixing. For the Hadamard graphs, we are able to show that no time $t$ exists which satisfy the conditions of Corollary \ref{cor:um}, except when $n=4$. In general, it would be enlightening to give a characterisation of instantaneous uniform mixing which does not depend on the time. 

\begin{openprob}
  Given the eigenvalues and eigenspaces of a graph, give necessary and sufficient conditions for uniform mixing to occur. 
\end{openprob}

We note that such a characterisation is likely difficult and related to the decidability of uniform mixing. It would be useful even in classes of graphs, like the class of distance-regular or strongly regular graphs. 


\section{Average mixing matrix}

Recall that
\[
M(t) = U(t) \circ U(-t)
\]
is doubly-stochastic. Its rows (and columns) are probability distributions, and each of its entries indicates the probability that certain special separable states evolve to others within the quantum walk. The typical behaviour of a quantum walk can therefore be inferred by taking averages of the matrices $M(t)$ for varying values of $t$. The uniform average on the interval $[0,T]$ is given by
\[
\widehat{M}_{[0,T]} = \frac{1}{T} \int_{0}^T M(t) \ \textrm{d}t.
\]
It is quite well-known that if $T \to \infty$, this average converges to a nice, special matrix, whose connection to the combinatorics of the graph has lead to unexpected results.

The average mixing matrix is defined as
\[
    \widehat{M} = \lim_{T \to\infty} \widehat{M}_{[0,T]},
\]
and a straightforward calculation shows that if $A(G) = \sum_{r = 0}^d \theta_r E_r$, then
\[
\widehat{M} = \sum_{r = 0}^d E_r \circ E_r.
\]
Averages of probability distribution coming from the rows of $M(t)$ were investigated in hypercubes in \cite{moore2002quantum}, complete graphs in \cite{adamczak2003note} and circulants in \cite{adamczak2007non}. Godsil \cite{GodsilAverageMixing} introduced the basic theory of $\widehat{M}$, and the connection of $\widehat{M}$ and average states and the commutant algebra of $A(G)$ was developed in \cite{CouGodGuo2018}.

As the theory of the average mixing matrix is relatively new, there are many interesting problems one can ask. We discuss a few below.

We begin with a problem that remotes to the topic of the previous section:

\begin{openprob}
    If $G$ is a primitive distance-regular graph, show that the rank of $\mxm$ is equal to the number of vertices.
\end{openprob}

\subsection{Different takes on the positivity of \texorpdfstring{$\widehat{M}$}{M}} 

It is quite clear that $\widehat{M}$ has nonnegative entries and is positive semidefinite. It was shown in \cite{CouGodGuo2018} that $\widehat{M}$ is in fact completely positive semidefinite, meaning, it is the Gram matrix of positive semidefinite matrices. In each of these cones of matrices, it is natural and often very useful to find the rank within the cone of a given matrix.

The non-negative rank of a non-negative $n\times n$ matrix $M$ is the least number $k$, such that there are $k$  matrices $\{M_r\}_{r=1}^k$ of rank 1 with non-negative entries, such that $M = \sum_{r=1}^k M_r$. 

The completely positive semidefinite rank (cpsd-rank) of a completely positive semidefinite matrix $M$ is the smallest dimension $d$ so that $M$ is the Gram matrix of positive semidefinite matrices of size $d\times d$.

The result in \cite{CouGodGuo2018} shows that the cpsd-rank of $\widehat{M}$ is at most $n$, and it is well-known that the non-negative rank of a non-negative $n\times n$ matrix is at most $n$. The paper \cite{CouGodGuo2018} also develops theory connecting the (usual) rank of $\widehat{M}$ and other properties of the graph, for example, there it is shown that if $G$ has simple eigenvalues and at least three vertices, then the rank of $\widehat{M}$ is at most $n-1$, with equality implying that all automorphisms of $G$ have fixed points. 

\begin{openprob}
    When is the non-negative rank of $\mxm$ is equal to $n$? 
\end{openprob}

\begin{openprob}
    Compute the cpsd-rank of $\mxm$.
\end{openprob}

\begin{openprob}
    In the extremal cases for the non-negative or cpsd ranks, what combinatorial consequences may be derived?
\end{openprob}

The matrix $\mxm$ is almost always completely positive. In other words, it is a Gram matrix of non-negative vectors.

\begin{lemma}
    If $G$ has simple eigenvalues, then $\widehat{M}$ is completely positive. 
\end{lemma}

\proof If the $\theta_r$-eigenspace has dimension $1$, then $E_r = \Zv_r \Zv_r^T$. Then we have 
\[
\widehat{M} = \sum_{r=0}^d E_r\circ E_r = \sum_{r=0}^d (\Zv_r \Zv_r^T)\circ(\Zv_r \Zv_r^T) = \sum_{r=0}^d (\Zv_r \circ\Zv_r)(\Zv_r \circ \Zv_r)^T.
\]
Since the entries of $\Zv_r\circ \Zv_r$ are square, they are non-negative and this gives a non-negative factorization of $\widehat{M}$.\qed

The completely positive rank (cp-rank) of a completely positive matrix $M$ is the smallest dimension $d$ for which $M$ is a Gram matrix of non-negative vectors in $\re^d$.

Since the rank of the average mixing matrix of a graph with simple eigenvalues is at most $n-1$, this may be a situation where it is possible to find the cp-rank or to characterise those whose cp-rank is equal to the number of vertices.

\begin{openprob}
    When is the cp-rank of $\mxm$ is equal to $n$? 
\end{openprob}

\subsection{Rank, trace, eigenvalues, and alternative averages}

The average mixing matrix is computed from the uniform average over intervals. There are other probability distributions for which the average of the mixing matrices lead to interesting phenomena. If $R$ is a random variable with probability density function $f_R(t): \mathbb{R} \xrightarrow{} [0, \infty)$, then the \emph{average mixing matrix under $R$} is defined as
\[\widehat{M}_{R} := \mathbb{E}[M(R)] = \int_{-\infty}^{\infty} M(t)f_R(t) \ \mathrm{d}t.\]
The recent pre-print \cite{baptista2023unexpected} explores this idea, and among other things, show that for the paths $P_3$ and $P_4$ there are choices of $R$ for which $\widehat{M}_{R}$ has rank $1$ (something that never happens for $\widehat{M}$ unless the graph is $P_2$). The natural question of whether this is also possible for $P_n$, $n \geq 5$, remains open, however here we point to another line of investigation.

Rank, diagonal entries, trace, eigenvalues, all of these very natural linear algebraic properties of a matrix have been studied for $\widehat{M}$ \cite{CouGodGuo2018,GodGuoSin2018,GodGuoSob2019}. Proposition 6 in \cite{baptista2023unexpected} provides absolute upper and lower bounds for these parameters in terms of those of $\widehat{M}$, for any $R$. Thus we propose the very open ended line of research:

\begin{openprob}
    Develop a spectral graph theory for the matrix $\widehat{M}_{R}$, $R$ an arbitrary random variable.
\end{openprob}

We point out that $R$ can always be chosen so that $\widehat{M}_{R} = M(\tau)$ for a fixed $\tau \in \re$, so developments to this theory above will address questions about state transfer or uniform mixing, so one should not be too ambitious. On the other hand, tools from probability theory should be useful, and have not been heavily explored in the context of this interplay between graph theory and quantum walks.

\section*{Acknowledgements}

Gabriel Coutinho acknowledges the support from CNPq. The collaboration of the two authors is supported by FAPEMIG grant APQ-04481-22. Krystal Guo acknowledges the support from FAPERJ.
Both authors would like to express our gratitude to Nair Abreu for her invaluable contributions to the spectral graph theory, and to the editors of this special issue for providing us with the opportunity to contribute to the celebration of Nair Abreu's birthday.

\printbibliography

\end{document}